\PassOptionsToPackage{numbers}{natbib}
\documentclass{winnower}

\usepackage[utf8]{inputenc} 
\usepackage[T1]{fontenc}    
\usepackage{hyperref}       
\usepackage{url}            
\usepackage{booktabs}       
\usepackage{amsfonts}       
\usepackage{nicefrac}       
\usepackage{microtype}      
\usepackage{amsmath,bm}
\usepackage{graphicx,epstopdf}
\usepackage{multirow}
\usepackage{caption}
\usepackage{subfigure}
\usepackage{comment}
\usepackage{float}
\usepackage{lscape}
\usepackage{listings}

\usepackage{color}
\usepackage{amssymb,dsfont}
\usepackage{ulem}
\usepackage{epstopdf}
\usepackage{leftidx}
\usepackage{multirow}
\usepackage{bigdelim}
\usepackage{mathrsfs}
\usepackage{blkarray,adjustbox}
\usepackage{lineno}
\usepackage{enumitem}
\usepackage{mathpazo}
\usepackage{booktabs} 
\usepackage{float}
\usepackage{amsmath}
\usepackage{hyperref}
\usepackage{amsthm}
\usepackage[dvipsnames]{xcolor}

\newtheorem{theorem}{Theorem}[section]
\newtheorem{proposition}{Proposition}[section]

\theoremstyle{definition}
\newtheorem{definition}{Definition}[section]

\theoremstyle{remark}
\newtheorem{Lemma}{Lemma}[section]

\newtheorem{Remark}{Remark}[section]

\newcommand{\vc}{\mathrm{vec}}
\newcommand{\T}{\mathsf{T}}
\newcommand{\tr}{\mathrm{tr}}
\newcommand{\dif}{\mathrm{d}}

\usepackage{soul}
\definecolor{reddish}{HTML}{FBB4AE}
\definecolor{blueish}{HTML}{B3CDE3}
\definecolor{magentish}{HTML}{FF00AA}
\definecolor{greenish}{HTML}{a1d99b}

\graphicspath{{./Figures/}} 

\begin{document}

\title{Remarks on multivariate Gaussian Process}

\author[1]{Zexun Chen \footnote{Corresponding author: Zexun Chen, Email: z.chen3@exeter.ac.uk
.}}

\author[2]{Jun Fan}
\author[3]{Kuo Wang}

\affil[1]{College of Engineering, Mathematics and Physical Sciences \\ 
University of Exeter, EX4 4QF, UK}
\affil[2]{College of Science, United Arab Emirates University, P.O. Box 15551, UAE}
\affil[3]{College of Mathematics, Physics and Information Engineering, Jiaxing university, 314033, China}

\date{}

\maketitle

\begin{abstract}
Gaussian processes occupy one of the leading places in modern statistics and probability theory due to their importance and a wealth of strong results. 
The common use of Gaussian processes is in connection with problems related to estimation, detection, and many statistical or machine learning models.
With the fast development of Gaussian process applications, 
it is necessary to consolidate the fundamentals of vector-valued stochastic processes, in particular multivariate Gaussian processes, which is the essential theory for many applied problems with multiple correlated responses.
In this paper, we propose a precise definition of multivariate Gaussian processes based on Gaussian measures on vector-valued function spaces, and provide an existence proof.
In addition, several fundamental properties of multivariate Gaussian processes, such as strict stationarity and independence, are introduced. 
We further derive multivariate Brownian motion including It\^o lemma as a special case of a multivariate Gaussian process, and present a brief introduction to multivariate Gaussian process regression as a useful statistical learning method for multi-output prediction problems.

\end{abstract}

\section{Introduction}

In the theory of stochastic processes, some general results on Gaussian processes play a essential role in the construction of Brownian motion, as they both arise naturally from the requirement of independent increments. Furthermore, an understanding of Gaussian processes also gives a better understanding of many fundamentals of stochastic analysis. These factors, together with the simplicity and wealth of important results in the field, have led Gaussian processes to be considered one of the outstanding sub-fields of modern statistics and probability theory.

Nowadays Gaussian processes (GP) are also often considered in the context of supervised machine learning that uses lazy learning and a measure of the similarity between points (the kernel function) to predict the value for an unseen point from training data. 
Rather than inferring a distribution over the parameters of an undetermined parametric function, GP can be used as a non-parametric model in order to infer a distribution over functions directly. 
A GP defines a prior over functions. 
Given some observed function values, it can achieve a posterior over functions. 
GP has been proven to be an effective method for nonlinear problems due to many desirable properties, such as a clear structure with Bayesian interpretation, a simple integrated approach of obtaining and expressing uncertainty in predictions and the capability of capturing a wide variety of data feature by hyper-parameters \cite{cite:RCE,boyle2005dependent}. 
Since \citet{cite:neal} revealed that many Bayesian neural networks converge to Gaussian processes in the limit of an infinite number of hidden units \cite{williams1997computing}, GP has been widely used as an alternative to neural networks in order to solve complicated regression and classification problems in many areas, e.g., Bayesian optimisation \cite{frazier2018tutorial}, time series forecasting \cite{brahim2004gaussian,mackay1997gaussian}, feature selection \cite{savitsky2011variable}, and so on.

With the development of Gaussian processes related to machine learning algorithms, the application of Gaussian processes has faced a conspicuous limitation. 
The classical GP model can be only used to deal with a single output or single response problem because the process itself is defined on $\mathbb{R}$, and as a result the correlation between multiple tasks or responses cannot be taken into consideration \cite{boyle2005dependent,wang2015gaussian}. 
In order to overcome the drawback above, many advanced Gaussian process model was proposed, including dependent Gaussian process \cite{boyle2005dependent}, Gaussian process regression with multiple response variables \cite{wang2015gaussian}, and Gaussian process regression for vector-valued function \cite{alvarez2011kernels}.
The general idea of these methods is to vectorise the multi-response variables and construct a "big" covariance, which describes the correlations between the inputs as well as between the outputs. 
Intrinsically, these approaches depend on the fact that the matrix-variate Gaussian distributions can be reformulated as multivariate Gaussian distributions, 
and these are still conventional Gaussian process regression models since the reformulation merely vectorises the multi-response variables, which are assumed to follow a developed case of GP with a reproduced kernel \cite{gupta1999matrix}.

In another development, \citet{chen2020multivariate} defined multivariate Gaussian processes (MV-GP) and proposed a unified framework to perform multi-output prediction using Gaussian processes.
This framework does not rely on the equivalence between vectorised matrix-variate Gaussian distribution and multivariate Gaussian distribution, and it can be easily used to produce a general elliptical process model, for example, multivariate Student-$t$ process (MV-TP) for multi-output prediction.
Both MV-GPR and MV-TPR have closed-form expressions for the marginal likelihoods and predictive distributions under this unified framework and thus can adopt the same optimization approaches as used in the conventional GP regression. 
Although \citet{chen2020multivariate} showed the usefulness of the proposed methods via data-driven examples, some theoretical issues of multivariate Gaussian processes are still not clear, e.g., the existence of MV-GP.

When it comes to the theoretical fundamentals of stochastic processes, a close look at measure theory is indispensable. 
Briefly speaking, (multivariate) Gaussian distributions are Gaussian measures on $\mathbb{R}^n$, and Gaussian processes are Gaussian measures on the function space $(\mathbb{R}_T, \mathcal{F})$ (for details refer to Definition~\ref{def:gm-on-Rn}, Definition~\ref{def:gp_measure}, and Theorem \ref{thm:relationshipGMandGP} below).
Based on the relationship between Gaussian measures and Gaussian processes, we properly defined multivariate Gaussian processes by extending Gaussian measures on function spaces to vector-valued function spaces.

The paper is organised as follows. 
Section~\ref{section:basic} introduces some preliminaries of Gaussian processes, including some useful properties and the proof of existence. 
Section~\ref{section:mv-gp} presents some theoretical definitions of multivariate Gaussian process with the proof of existence. 
The examples and application of multivariate Gaussian processes which show their usefulness is presented in Section~\ref{section:example} and Section~\ref{section:application}. 
Conclusions and a discussion are given in Section~\ref{section:conclusion}.

\section{Preliminary of Gaussian process}\label{section:basic}
\subsection{Stochastic process}
A stochastic (or random) process is defined by a collection of random variables defined on a common probability space $(\Omega ,{\mathcal{F}},\mathcal{P})$, where $\Omega$ is a sample space, $\mathcal {F}$ is a $\sigma$-algebra and $\mathcal{P}$ is a probability measure; and the random variables, indexed by some set $T$, all take values in the same mathematical space $S$, which must be measurable with respect to some $\sigma$ -algebra 
$\Sigma$ \cite{lamperti2012stochastic}.
In other words, for a given probability space $(\Omega ,{\mathcal{F}},\mathcal{P})$ and a measurable space $(S,\Sigma )$, a stochastic process is a collection of $S$-valued random variables, which can be written as:
$$
\{f(t):t\in T\}.
$$

A stochastic process can be interpreted or defined as a $S_T$-valued random variable, where $S_T$ is the space of all the possible $S$-valued functions of $t\in T$ that map from the set $T$ into the space $S$ \cite{kallenberg2006foundations}. 
The set $T$ is usually one of these: 
$$
\mathbb{R}, \mathbb{R}^{n}, \mathbb{R}^{+} = [0, +\infty), \mathbb{Z} = \{\cdots, -1, 0, 1, \cdots \}, \mathbb{Z}^{+} = (0, 1, \cdots).
$$
If $T = \mathbb{Z}$ or $ \mathbb{Z}^{+}$, we always call it \emph{random sequence}. If $T = \mathbf{R}^n$ with $n >1$, the process is often considered as a \emph{random field}.
The set $S$ is called \emph{state space} and usually formulated as one of these:
$$
\mathbb{R}, \{0,1\}, \mathbb{Z}^{+}, \mathbb{D} = \{A, B, C, \cdots \}.
$$
Indeed, the random variable of the process is not required to be in the form of one of these above sets, but it must have the same measurable $S$. For example, if $S = \mathbb{R}^{d}$ or $\mathbb{D}^d$ with $d>1$, it is called \emph{vector-valued} process.


\subsection{Gaussian measure and distribution}

\begin{definition}[Gaussian measure on $\mathbb{R}$]
Let $\mathcal{B}(\mathbb{R})$ denote the completion of the Borel $\sigma$-algebra on $\mathbb{R}$. Let $\lambda: \mathcal{B}(\mathbb{R}) \mapsto [0, + \infty]$ denote the usual Lebesgue measure. Then the Borel probability measure $\gamma: \mathcal{B}(\mathbb{R}) \mapsto [0, 1] $ is Gaussian with mean $\mu \in \mathbb{R}$ and variance $\sigma^2 >0$,
$$
\gamma(A) = \int_{A} \frac{1}{\sqrt{2\pi \sigma^2}}\exp \left(-\frac{(x - \mu)^2}{2\sigma^2}  \right) \dif \lambda(x)
$$
for any measurable set $A \in \mathcal{B}(\mathbb{R})$.
\end{definition} 
A random variable $X$ on a probability space $(\Omega, \mathcal{F}, \mathcal{P})$ is Gaussian with mean $\mu$ and variance $\sigma^2$ if its distribution measure is Gaussian, i.e.
$$
\mathcal{P}(X \in A) = \gamma(A)
$$
As we know from the view of random variable, we have 
\begin{definition}
  An $n$-dimensional random vector $\bm{X} = (X_1, \cdots, X_n)$ is Gaussian if and only if $\langle \bm{a}, \bm{X} \rangle :  = \bm{a}^{\T} \bm{X} = \sum a_i X_i$ is a Gaussian random variable for all $\bm{a} = (a_1, \cdots, a_n) \in \mathbb{R}^n$.
\end{definition}

In terms of measure,
\begin{definition}[Gaussian measure on $\mathbb{R}^n$]
\label{def:gm-on-Rn}
 Let $\gamma$ be a Borel probability measure on $\mathbb{R}^n$. For each $\bm{a} \in \mathbb{R}^n$, denote a random variable $Y(\bm{x} \in \mathbb{R}^n)$ as a mapping $\bm{x} \mapsto \langle \bm{a}, \bm{x} \rangle \in \mathbb{R}$ on the probability space $(\mathbb{R}^n, \mathcal{B}(\mathbb{R}^n), \gamma)$. The Borel probability measure $\gamma$ is a Gaussian measure on $\mathbb{R}^n$ if and only if the random variable $Y$ is Gaussian for each $\bm{a}$.
\end{definition}

A matrix Gaussian distribution in statistics is a probability distribution by generalizing the multivariate normal distribution to matrix-valued random variables, which can be defined by multivariate Gaussian distribution.
\begin{definition}[Matrix Gaussian distribution]\label{def:matrix-gm}
The random matrix is said to be Gaussian \cite{gupta1999matrix}:
$$
\bm{X} \sim \mathcal{MN}_{n, d}(M, U, V),
$$
if and only if 
$$
\vc(\bm{X}) \sim \mathcal{N}_{nd}(\vc(M), V \otimes U),
$$
where $\otimes$ denotes the Kronecker products and $\vc(\bm{X})$ denotes the vectorisation of $\bm{X}$.
\end{definition}

\begin{theorem}[Marginalization and conditional distribution \cite{chen2020multivariate, gupta1999matrix}]\label{thm:MarginCondition-G}
  Let 
  $$X \sim \mathcal{MN}_{n,d}(M, \Sigma, \Lambda)$$
  and partition $X, M, \Sigma$ and $\Lambda$ as
\begin{equation*}
  X =
  \begin{adjustbox}{raise=-1ex} $\displaystyle
    \begin{blockarray}{[c]c}
      X_{1r} & n_1\\
      X_{2r} & n_2
    \end{blockarray} $
  \end{adjustbox}
  =
    \begin{adjustbox}{raise=-1ex} $\displaystyle
    \begin{blockarray}{cc}
          &            \\
      \begin{block}{[cc]}
         X_{1c} & X_{2c} \\
      \end{block}
      d_1 & d_2
    \end{blockarray}$
  \end{adjustbox}
  ,\quad
  M =
  \begin{adjustbox}{raise=-1ex} $\displaystyle
    \begin{blockarray}{[c]c}
      M_{1r} & n_1\\
      M_{2r} & n_2
    \end{blockarray} $
  \end{adjustbox}
  =
    \begin{adjustbox}{raise=-1ex} $\displaystyle
    \begin{blockarray}{cc}
          &        \\
      \begin{block}{[cc]}
         M_{1c} & M_{2c} \\
      \end{block}
      d_1 & d_2
    \end{blockarray} $
  \end{adjustbox}
\end{equation*}
\begin{equation*}
  \Sigma =
  \begin{adjustbox}{raise=-2.5ex} $\displaystyle
    \begin{blockarray}{ccc}
      \begin{block}{[cc]c}
        \Sigma_{11} & \Sigma_{12} &n_1\\
        \Sigma_{21} & \Sigma_{22} & n_2\\
      \end{block}
      n_1 & n_2 &
    \end{blockarray} $
  \end{adjustbox}
  \quad \text{and} \quad
  \Lambda =
  \begin{adjustbox}{raise=-2.5ex} $\displaystyle
    \begin{blockarray}{ccc}
      \begin{block}{[cc]c}
        \Lambda_{11} & \Lambda_{12} & d_1\\
        \Lambda_{21} & \Lambda_{22} & d_2\\
      \end{block}
      d_1 & d_2 &
    \end{blockarray} $
  \end{adjustbox},
\end{equation*}
where $n_1,n_2, d_1,d_2$ is the column or row length of the corresponding vector or matrix. Then,
\begin{enumerate}[leftmargin=*,labelsep=3mm]
   \item $X_{1r} \sim \mathcal{MN}_{n_1,d}\left(M_{1r},\Sigma_{11},\Lambda \right)$,
   $$
   X_{2r}|X_{1r} \sim \mathcal{MN}_{n_2,d}\left(M_{2r} + \Sigma_{21}\Sigma_{11}^{-1}(X_{1r}-M_{1r}),\Sigma_{22\cdot1},\Lambda \right);
   $$
   \item $X_{1c} \sim \mathcal{MN}_{n,d_1}\left(M_{1c},\Sigma,\Lambda_{11}\right)$,
   $$
       X_{2c}|X_{1c} \sim \mathcal{MN}_{n,d_2}\left(M_{2c} + (X_{1c}-M_{1c})\Lambda_{11}^{-1}\Lambda_{12},\Sigma,\Lambda_{22\cdot1} \right);
   $$
\end{enumerate}
where $\Sigma_{22\cdot1} $and $\Lambda_{22\cdot1}$ are the Schur complement \cite{zhang2006schur} of $\Sigma_{11}$ and $\Lambda_{11}$, respectively,
$$
\Sigma_{22\cdot1} = \Sigma_{22} - \Sigma_{21}\Sigma_{11}^{-1}\Sigma_{12} , \quad \Lambda_{22\cdot1} = \Lambda_{22} - \Lambda_{21}\Lambda_{11}^{-1}\Lambda_{12}.
$$
\end{theorem}

\subsection{Gaussian process}

Consider the space $\mathbb{R}_T$ of all $\mathbb{R}$-valued functions on $T$. A subset of the form $\{f: f(t_i) \in A_i, 1 \leq i \leq n \}$ for some $n \geq 1, t_i \in T$ and some Borel sets $A_{i} \subseteq \mathbb{R}$ is called a \emph{cylinder} set. 
Let $\mathcal{F}$ be the $\sigma$-algebra generated by all cylinder sets.

Also, we may consider the product topology on $\mathbb{R}_T$, which defined as the smallest topology that makes the projection maps $\Pi_{t_{1}, \ldots, t_{n}}(f)=\left[f\left(t_{1}\right), \ldots, f\left(t_{n}\right)\right]$ from $\mathbb{R}_T$ to
$\mathbb{R}^n$ measurable, and define $\mathcal{F}$ as the Borel $\sigma$-algebra of this topology. We can obtain: 

\begin{definition}[Gaussian measure on $(\mathbb{R}_T, \mathcal{F})$]\label{def:gp_measure}
A measure $\gamma$ on $(\mathbb{R}_T, \mathcal{F})$ is called as a Gaussian measure if for any $n \geq 1$ and $t_1, \cdots, t_n \in T$, the push-forward measure $\gamma \circ \Pi_{t_1, \cdots, t_n}^{-1}$ on $\mathbb{R}^n$ is a Gaussian measure.
\end{definition}

\begin{theorem}[Relationship between Gaussian process and Gaussian measure]\label{thm:relationshipGMandGP}
If $X=(X_t)_{t \in T}$ is a Gaussian process, then the push-forward measure $\gamma = \mathcal{P} \circ X^{-1}$ with $X: \Omega \mapsto \mathbb{R}_{T}$ is Gaussian on $\mathbb{R}_{T}$, namely, $\gamma$ is a Gaussian measure on $(\mathbb{R}_T, \mathcal{F})$. Conversely, if $\gamma$ is a Gaussian measure on $(\mathbb{R}_T, \mathcal{F})$, then on the probability space $(\mathbb{R}_T, \mathcal{F}, \gamma)$, the co-ordinate random variable $\Pi = (\Pi_t)_{t \in T}$ is from a Gaussian process.
\end{theorem}

The proof of the relationship between Gaussian process and Gaussian measure can be found in \cite{rajput1972gaussian}.


\begin{theorem}[Existence of Gaussian process]\label{thm:gp-existence}
For any index set $T$, any mean function $\mu: T \mapsto \mathbb{R}$ and any covariance function (function has covariance form), $k: T \times T \mapsto \mathbb{R}$, there exists a probability space $(\Omega, \mathcal{F}, \mathcal{P})$ and a Gaussian process $\mathcal{GP}(\mu, k)$ on this space, whose mean function is $\mu$ and covariance function is $k$. It is denoted as $X \sim \mathcal{GP}(\mu, k)$.
\end{theorem}


\begin{proof}
Thanks to Theorem~\ref{thm:relationshipGMandGP}, we just need to prove the existence of Gaussian measure with the specific mean vector generated by mean function and specific covariance matrix generated by covariance function. 
Given $n > 1$, for every $t_1, \cdots, t_n \in T$, a Gaussian measure $\gamma_{t_1, \cdots, t_n}$ on $\mathbb{R}^n$ satisfies the assumptions of \emph{Daniell-Kolmogorov theorem} because the projection of Gaussian distribution on $\mathbb{R}^n$ with $n$-dimensional vector $\left[\mu(t_1), \cdots, \mu(t_n)\right] \in \mathbb{R}^n$ and $n \times n$ covariance matrix $K = (k_{i,j}) \in \mathbb{R}^{n \times n}$, to the first $n-1$ co-ordinates, is precisely a Gaussian distribution with $n-1$-dimensional vector $\left[\mu(t_1), \cdots, \mu(t_{n-1}) \right] \in \mathbb{R}^{n-1}$ and $(n-1) \times (n-1)$ covariance matrix $K = (k_{i,j}) \in \mathbb{R}^{(n-1) \times (n-1)}$. By the Daniell-Kolmogorov theorem, there exists a probability space $(\Omega, \mathcal{F}, \mathcal{P})$ as well as a Gaussian process $X = (X_t)_{t \in T} \sim \mathcal{GP}(\mu, k)$ defined on this space such that any finite dimensional distribution of $\left[X_{t_1}, \cdots, X_{t_n}\right]$ is given by the measure $\gamma_{t_1, \cdots, t_n}$.
\end{proof}

\section{Multivariate Gaussian process}\label{section:mv-gp}

Following the classical theory of Gaussian measure and Gaussian process, we can introduce Gaussain measure on $\mathbb{R}^{n \times d}$ and Gaussian measure on $((\mathbb{R}^n)_T, \mathcal{G})$, and finally define the multivariate Gaussian process.

According to Definition~\ref{def:gm-on-Rn} and Definition~\ref{def:matrix-gm}, we can have a definition of Gaussian measure on $\mathbb{R}^{n \times d}$.



\begin{definition}[Gaussian measure on $\mathbb{R}^{n \times d}$]
 Let $\gamma$ be a Borel probability measure on $\mathbb{R}^{n \times d}$. For each $\bm{a} \in \mathbb{R}^{nd}$, denote a random variable $Y(\bm{x} \in \mathbb{R}^{n \times d})$ as a mapping $\bm{x} \mapsto \langle \bm{a}, \vc(\bm{x}) \rangle \in \mathbb{R}$ on the probability space $(\mathbb{R}^{n \times d}, \mathcal{B}(\mathbb{R}^{n \times d}), \gamma)$. The Borel probability measure $\gamma$ is a Gaussian measure on $\mathbb{R}^{n \times d}$ if and only if the random variable $Y$ is Gaussian for each $\bm{a}$.
\end{definition}

Similarly to the introduction in Gaussian process, now we consider the space $(\mathbb{R}^{d})_T$ of all $\mathbb{R}^{d}$-valued functions on $T$. 
Let $\mathcal{G}$ be a $\sigma$-algebra generated by all cylinder sets where each cylinder set here defined as a subset of the form $\{\bm{f}: \bm{f}(t_i) \in B_i, 1 \leq i \leq n \}$ for some $n \geq 1, t_i \in T$ and some Borel sets $B_{i} \subseteq \mathbb{R}^{d}$. 
Also, we can define the smallest topology on $\mathbb{R}^d$ that makes the projection mappings $\Xi_{t_{1}, \ldots, t_{n}}(f)=\left[f\left(t_{1}\right), \ldots, f\left(t_{n}\right)\right]$ from $(\mathbb{R}^{d})_T$ to
$\mathbb{R}^{n \times d}$ measurable, and define $\mathcal{G}$ as the Borel $\sigma$-algebra of this topology.
Thus we can have a definition of Gaussian measure on $((\mathbb{R}^d)_T, \mathcal{G})$.

\begin{definition}[Gaussian measure on $((\mathbb{R}^d)_T, \mathcal{G})$]
A measure $\gamma$ on $((\mathbb{R}^d)_T, \mathcal{G})$ is called as a Gaussian measure if for any $n \geq 1$ and $t_1, \cdots, t_n \in T$, the push-forward measure $\gamma \circ \Xi_{t_1, \cdots, t_n}^{-1}$ on $\mathbb{R}^{n \times d}$ is a Gaussian measure. 
\end{definition}

Since the relationship between Gaussian process and Gaussian measure in  Theorem~\ref{thm:relationshipGMandGP}, we can well define \emph{\textbf{multivariate Gaussian process} (MV-GP)}.

\begin{definition}[$d$-variate Gaussian process]
Given a Gaussian measure on $((\mathbb{R}^d)_T, \mathcal{G})$, $d \geq 1$, the co-ordinate random vector $\Xi = (\Xi_t)_{t \in T}$ on the probability space $((\mathbb{R}^d)_T, \mathcal{G}, \gamma)$ is said to be from a \emph{$d$-variate Gaussian process}.
\end{definition}

\begin{theorem}[Existence of $d$-variate Gaussian process] \label{thm:d-variateGP}
For any index set $T$, any vector-valued mean function $\bm{u} : T \mapsto \mathbb{R}^d$, 
any covariance function $k: T \times T \mapsto \mathbb{R}$ and any positive semi-definite parameter matrix $\Lambda \in \mathbb{R}^{d \times d}$, 
there exists a probability space $(\Omega, \mathcal{G}, \mathcal{P})$ and a $d$-variate Gaussian process $\bm{f}(x)$ on this space, whose mean function is $\bm{u}$, covariance function is $k$ and parameter matrix is $\Lambda$ 
, such that,
\begin{itemize}
    \item $\mathbb{E}[\bm{f}(t)] = \bm{u}(t), \quad \forall t \in T,$
    \item $\mathbb{E}\left[(\bm{f}(t_s)- \bm{u}(t_s))(\bm{f}(t_l)- \bm{u}(t_l))^{\T}\right] = \tr(\Lambda)k(t_s,t_l),\quad \forall t_s,t_l \in T$
    \item 
    $\mathbb{E}[(\textbf{F}_{t_1, \cdots, t_n} - M_{t_1, \cdots, t_n})^\T (\textbf{F}_{t_1, \cdots, t_n} - M_{t_1, \cdots, t_n}) ] = \tr(K_{t_1, \cdots, t_n}) \Lambda 
    , \quad \forall n\geq 1, t_1, \cdots, t_n \in T$, where
    \begin{align*}
        M_{t_1, \cdots, t_n} &= [\bm{u}(t_1)^\T, \cdots, \bm{u}(t_n)^\T]^\T \\
        F_{t_1, \cdots, t_n} &= [\bm{f}(t_1)^\T, \cdots, \bm{f}(t_n)^\T]^\T \\
        K_{t_1, \cdots, t_n} &= 
        \begin{bmatrix}
        k(t_1, t_1)      & \cdots    & k(t_1, t_n)      \\
        \vdots  &  \ddots   & \vdots \\
        k(t_n, t_1)     & \cdots    & k(t_n, t_n)
        \end{bmatrix}
    \end{align*}
\end{itemize}
It denotes $\bm{f} \sim \mathcal{MGP}_d(\bm{u}, k, \Lambda )$.
\end{theorem}

\begin{proof}
Given $n > 1$, for every $t_1, \cdots, t_n \in T$, a Gaussian measure $\gamma_{t_1, \cdots, t_n}$ on $\mathbb{R}^{n \times d}$ satisfies the assumptions of \emph{Daniell-Kolmogorov theorem} because the projection of a matrix Gaussian distribution on $\mathbb{R}^{n \times d}$ with $\left[\bm{u}(t_1)^{\T}, \cdots, \bm{u}(t_n)^{\T}\right]^{\T} \in \mathbb{R}^{n \times d}$, $n \times n$ column covariance matrix $K = (k_{i,j}) \in \mathbb{R}^{n \times n}$, and $d \times d$ row covariance matrix $\Lambda  \in \mathbb{R}^{d \times d}$,
to the first $n-1$ co-ordinates, is precisely the Gaussian distribution with 
$\left[\bm{u}(t_1)^{\T}, \cdots, \bm{u}(t_{n-1})^{\T} \right]^{\T} \in \mathbb{R}^{(n-1) \times d}$, $(n-1) \times (n-1)$ column covariance matrix $K = (k_{i,j}) \in \mathbb{R}^{(n-1) \times (n-1)}$, and row covariance matrix $\Lambda  \in \mathbb{R}^{d \times d}$.
This is due to the conditional property of matrix Gaussian distribution shown in Theorem~\ref{thm:MarginCondition-G}.
By the Daniell-Kolmogorov theorem, there exists a probability space $(\Omega, \mathcal{G}, \mathcal{P})$ as well as a $d$-variate Gaussian process $X = (X_t)_{t \in T} \sim \mathcal{MGP}_d(\bm{u}, k, \Lambda )$ defined on this space such that any finite dimensional distribution of $[X_{t_1}, \cdots, X_{t_n}]$ is given by the measure $\gamma_{t_1, \cdots, t_n}$.
\end{proof}

Following the existence of $d$-variate Gaussian process, we can also achieve some properties as follow.
\begin{proposition}[Strictly stationary]
\label{prop:stationary}

A $d$-variate Gaussian process $\mathcal{MGP}_d(\bm{u}, k, \Lambda)$ is said to be strictly stationary if 
\begin{equation*}
    \bm{u}(t) = \bm{u}(t + h), \quad
    k(t_s + h, t_l + h) = k(t_s, k_l), \forall t, t_s, t_l, h \in T. 
\end{equation*}
\end{proposition}

\begin{proof}
Assume $\bm{f} \sim \mathcal{MGP}_d(\bm{u}, k, \Lambda)$, then for $\forall n\geq 1, t_1, \cdots, t_n \in T$, 
$$
[\bm{f}(t_1)^{\mathrm{T}},\ldots,\bm{f}(t_n)^{\mathrm{T}}]^{\mathrm{T}} \sim \mathcal{MN}(\bm{u}_{t_1, \cdots, t_n},K_{t_1, \cdots, t_n},\Lambda),
$$
where
\begin{equation*}
    \bm{u}_{t_1, \cdots, t_n} = \begin{bmatrix}
        \bm{u} (t_1)      \\
        \vdots  \\
        \bm{u} (t_n)  
        \end{bmatrix} , \quad
    K_{t_1, \cdots, t_n} = 
        \begin{bmatrix}
        k(t_1, t_1)      & \cdots    & k(t_1, t_n)      \\
        \vdots  &  \ddots   & \vdots \\
        k(t_n, t_1)     & \cdots    & k(t_n, t_n)
        \end{bmatrix}
\end{equation*}

Given any time increment $h \in T$, there also exists, 
$$
[\bm{f}(t_1 + h)^{\mathrm{T}},\ldots,\bm{f}(t_n + h)^{\mathrm{T}}]^{\mathrm{T}} \sim \mathcal{MN}(\bm{u}_{t_1+h, \cdots, t_n+h},K_{t_1+h, \cdots, t_n+h},\Lambda),
$$
where
\begin{equation*}
    \bm{u}_{t_1+h, \cdots, t_n + h} = \begin{bmatrix}
        \bm{u}(t_1 +h )      \\
        \vdots  \\
        \bm{u} (t_n +h )  
        \end{bmatrix} , \quad
    K_{t_1 + h, \cdots, t_n + h} = 
        \begin{bmatrix}
        k(t_1 + h, t_1 + h)      & \cdots    & k(t_1 + h, t_n + h)      \\
        \vdots  &  \ddots   & \vdots \\
        k(t_n + h, t_1 + h)     & \cdots    & k(t_n + h, t_n + h)
        \end{bmatrix}.
\end{equation*}

Since $\bm{u}(t) = \bm{u}(t + h),
    k(t_s + h, t_l + h) = k(t_s, k_l), \forall t, t_s, t_l, h \in T$, 
\begin{equation*}
    \bm{u}_{t_1+h, \cdots, t_n + h} = \begin{bmatrix}
        \bm{u}(t_1 +h)      \\
        \vdots  \\
        \bm{u} (t_n +h)  
        \end{bmatrix} 
        =
        \begin{bmatrix}
        \bm{u} (t_1)      \\
        \vdots  \\
        \bm{u} (t_n)  
        \end{bmatrix} 
        = \bm{u}_{t_1, \cdots, t_n }, 
\end{equation*}
\begin{equation*}
    K_{t_1 + h, \cdots, t_n + h} = 
        \begin{bmatrix}
        k(t_1 + h, t_1 + h)      & \cdots    & k(t_1 + h, t_n + h)      \\
        \vdots  &  \ddots   & \vdots \\
        k(t_n + h, t_1 + h)     & \cdots    & k(t_n + h, t_n + h)
        \end{bmatrix} 
        = 
        \begin{bmatrix}
        k(t_1 , t_1)      & \cdots    & k(t_1, t_n)      \\
        \vdots  &  \ddots   & \vdots \\
        k(t_n, t_1)     & \cdots    & k(t_n, t_n)
        \end{bmatrix} 
        = K_{t_1, \cdots, t_n}.
\end{equation*}
Therefore, $[\bm{f}(t_1 + h)^{\mathrm{T}},\ldots,\bm{f}(t_n + h)^{\mathrm{T}}]^{\mathrm{T}} $ has the same distribution as $[\bm{f}(t_1)^{\mathrm{T}},\ldots,\bm{f}(t_n)^{\mathrm{T}}]^{\mathrm{T}} $. Due to the arbitrary choice of $n>1$ and $t_1, \cdots, t_n \in T$, $\bm{f} \sim \mathcal{MGP}_d(\bm{u}, k, \Lambda)$ is a strictly stationary process.
\end{proof}

\begin{proposition}[Independence]
\label{Prop:independence}
A $d$ collection of functions $\{\bm{f}_i\}_{i = 1,2, \cdots, d}$ identically independently follows a Gaussian process $\mathcal{GP}(\mu, k)$ if and only if 
\begin{equation*}
\bm{f} = [\bm{f}_1, \bm{f}_2, \cdots, \bm{f}_d] \sim \mathcal{MGP}_d(\bm{u}, k, \Lambda),   
\end{equation*}
where $\bm{u} = [\mu, \cdots, \mu] \in \mathbb{R}^d$ and $\Lambda$ is any diagonal positive semi-definite matrix.
\end{proposition}

\begin{proof}
Necessity: if $\bm{f} \sim \mathcal{MGP}_d(\bm{u}, k, \Lambda)$, then for $\forall n\geq 1, t_1, \cdots, t_n \in T$, 
$$
[\bm{f}(t_1)^{\mathrm{T}},\ldots,\bm{f}(t_n)^{\mathrm{T}}]^{\mathrm{T}} \sim \mathcal{MN}(\bm{u}_{t_1, \cdots, t_n},K_{t_1, \cdots, t_n},\Lambda),
$$
where, 
\begin{equation*}
    \bm{u}_{t_1, \cdots, t_n} = \begin{bmatrix}
        \bm{u} (t_1)      \\
        \vdots  \\
        \bm{u} (t_n)  
        \end{bmatrix} , \quad
    K_{t_1, \cdots, t_n} = 
        \begin{bmatrix}
        k(t_1, t_1)      & \cdots    & k(t_1, t_n)      \\
        \vdots  &  \ddots   & \vdots \\
        k(t_n, t_1)     & \cdots    & k(t_n, t_n)
        \end{bmatrix}
\end{equation*}
Rewrite the left, we obtain
$$
[\bm{\xi}_1, \bm{\xi}_2, \cdots, \bm{\xi}_d] \sim \mathcal{MN}(\bm{u}_{t_1, \cdots, t_n},K_{t_1, \cdots, t_n},\Lambda),
$$
where $\bm{\xi}_i = [f_i(t_1), f_i(t_2), \cdots, f_i(t_n)]^{\T}$. Since $\Lambda$ is a diagonal matrix, for any $i \neq j$
$$
\mathbb{E}[\bm{\xi}_i^{\T} \bm{\xi}_{j}] = \tr( K_{t_1, \cdots, t_n}) \Lambda_{ij} = \tr( K_{t_1, \cdots, t_n}) \cdot 0 = 0.
$$
Because $\bm{\xi}_i$ and $\bm{\xi}_j$ are any finite number of realisations of $\bm{f}_i$ and $\bm{f}_j$ respectively from the same Gaussian process, $\bm{f}_i$ and $\bm{f}_j$ are uncorrelated. Due to  joint finite realisations of $\bm{f}_i$ and $\bm{f}_j$ follow Gaussian, non-correlation implies independence.

Sufficiency: if $\{\bm{f}_i\}_{i = 1,2, \cdots, d} \sim \mathcal{GP}(0, k)$ are independent, for $\forall n\geq 1, t_1, \cdots, t_n \in T$ and for any $i \neq j$,
$$
0 = \mathbb{E}[\bm{\xi}_i^{\T} \bm{\xi}_{j}] = \tr( K_{t_1, \cdots, t_n}) \Lambda_{ij}.
$$
Since $\tr( K_{t_1, \cdots, t_n})$ is non-zero, $\Lambda_{ij}$ must be 0. Due the arbitrary choices of $i, j$, $\Lambda$ must be diagonal.
That is to say, $\bm{\xi}_i = [f_i(t_1), f_i(t_2), \cdots, f_i(t_n)]^{\T}$ can be written as a matrix Gaussian distribution $\mathcal{MN}(\bm{u}_{t_1, \cdots, t_n}, K_{t_1, \cdots, t_n}, \Lambda)$ where $\Lambda$ is a diagonal positive semi-definite matrix. Since for $\forall n\geq 1, t_1, \cdots, t_n \in T$ we hold the above result, $\{\bm{f}_i\}_{i = 1,2, \cdots, d}$ can be considered identically independently Gaussian process $\mathcal{GP}(\mu, k)$. 
\end{proof}

\section{Example: special cases}\label{section:example}
Instinctively, a special case is \emph{centred multivariate Gaussian process} where vector-valued mean function $\bm{\mu} = \bm{0}$. The 50 realisation samples generated from centred multivariate Gaussian process are demonstrated in \autoref{fig:MV-GP-sample}:Left. 
Furthermore, we can derive the multivariate Gaussian white noise and the multivariate Brownian motion.

\begin{figure}[htbp]
	\centering
	\includegraphics[width=0.49\linewidth]{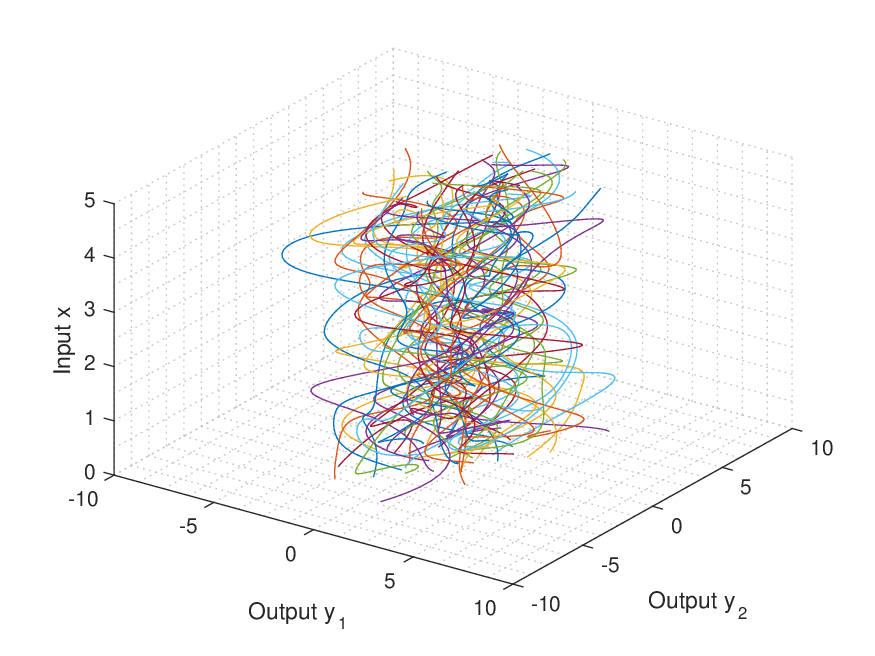}
	\includegraphics[width=0.49\linewidth]{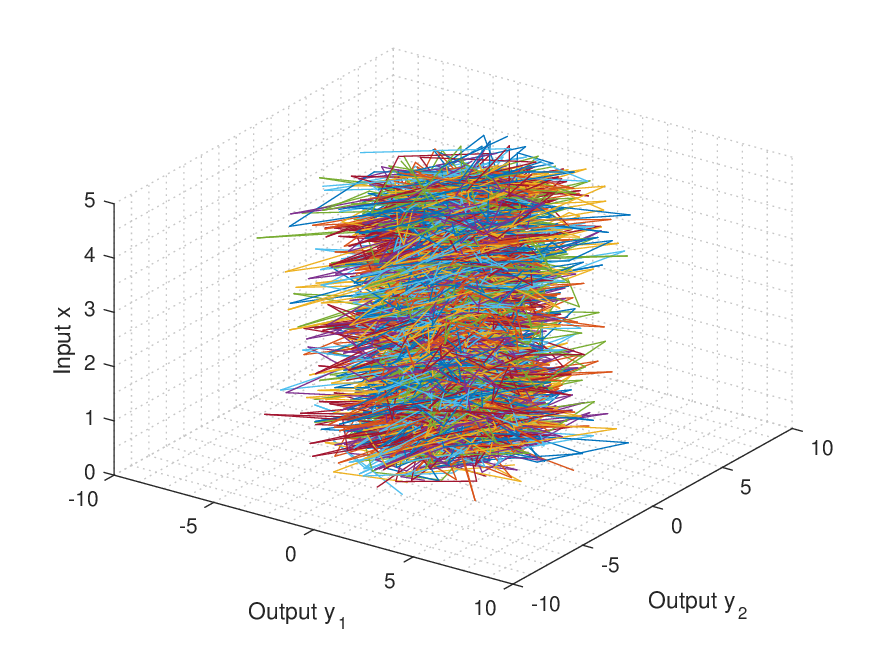}
	\caption{\textbf{The 50 random realisation sample points generated from of 2-variate Gaussian process}. 
	\textbf{Left}: centred 2-variate Gaussian process with Gaussian covariance function $k(t_s, t_l) = 1.5 \exp(-(t_s - t_l)^2/2/0.5^2)$. 
	\textbf{Right}: 2-variate Gaussian white noise as a 2-variate Gaussian process with covariance function $k(t_s, t_l) = 1.5\delta(t_s, t_l)$.
	}
	\label{fig:MV-GP-sample}
\end{figure}

\subsection{Multivariate Gaussian white noise}

\begin{proposition}
[$d$-variate Gaussian white noise]
A $d$-variate Gaussian process $\mathcal{MGP}_d (\bm{u}, k, \Lambda)$ is said to be $d$-variate Gaussian white noise if $\bm{u} = \bm{0}$ and $k(t_s, t_l) = \sigma^2 \delta(t_s - t_l)$, where $\delta$ is Dirac delta function and $t_s, t_l \in T$.
\end{proposition}

\begin{proof}
Let $\bm{f} = [\bm{f}_1,\cdots, \bm{f}_d ] \sim \mathcal{MGP}_d(\bm{u}, k, \Lambda)$, then 
$\mathbb{E}[\bm{f}(t)] = \bm{u}(t) = \bm{0}, \quad \forall t \in T$, and 
\begin{equation*}
    \mathbb{E}\left[(\bm{f}(t_s)- \bm{u}(t_s))(\bm{f}(t_l)- \bm{u}(t_l))^{\T}\right] = \tr(\Lambda)k(t_s,t_l) = \begin{cases}
    0 & \mbox{if }t_s \neq t_l \\
    \sigma^2\tr(\Lambda) & \mbox{if }t_s = t_l 
    \end{cases}
    ,\quad \forall t_s,t_l \in T
\end{equation*}
Furthermore, $\forall n\geq 1, t_1, \cdots, t_n \in T$, there exists, 
\begin{equation*}
    \mathbb{E}[\textbf{F}_{t_1, \cdots, t_n}^\T \textbf{F}_{t_1, \cdots, t_n}] = \tr(K_{t_1, \cdots, t_n}) \Lambda = \tr(\sigma^2 \mathrm{I}_{d \times d})\Lambda = d \sigma^2 \Lambda,
\end{equation*}
where $F_{t_1, \cdots, t_n} = [\bm{f}(t_1)^\T, \cdots, \bm{f}(t_n)^\T]^\T$.
\end{proof}

\begin{Remark} We observed in the proof that $d$-variate Gaussian white noise has independence property as white noise along with $T$, but it has correlation along with $d$-variate dimension. Therefore, $d$-variate Gaussian white noise is also called as variate-dependent Gaussian white noise or variate-correlated Gaussian white noise, which is distinct from the traditional $d$-dimensional independent Gaussian white noise. Here are 50 realisation samples generated from multivariate Gaussian white noise shown in \autoref{fig:MV-GP-sample}:Right. 
\end{Remark}


\subsection{Multivariate Brownian motion}
\label{MBM}
According to the Chapter 2 of the book written by \citet{cite:Gall}, there is a definition of Brownian motion, which is a Gaussian white noise whose intensity is Lebesgue measure.  
Since Brownian motion is a special case of Gaussian process with continuous sample paths, mean function $u = 0$ and covariance function $k(s,t) = \min(s,t)$, we propose an example, which is $d$-variate Brownian motion, as a special case of $d$-variate Gaussian process with vector-valued mean function $\bm{u} = \bm{0}$, covariance function $k(s,t) = \min(s,t)$ and parameter matrix $\Lambda$. 
Based on the Theorem \ref{thm:d-variateGP}, we derived some properties of the traditional Brownian motion to a more general vector-valued case. 

\begin{definition}[$d$-variate Brownian motion]
A $d$-variate Gaussian process $\mathcal{MGP}_d(\bm{u}, k, \Lambda)$ is said to be $d$-variate Brownian motion if all sample paths are continuous, $\bm{u} = \bm{0}$ and $k(t_s, t_l) =\min (t_s - t_l)$.
\end{definition}
Let $B_{t}$ be a $d$-variate Brownian motion, which means for all $0 \leq t_1 \leq \cdots \leq t_n$ the random variable $Z = (B_{t_1}^{\T}, \ldots, B_{t_n}^{\T})^{\T} \in \mathbb{R}^{n \times d}$ has a normal distribution on the probability space $(\Omega, \mathcal{G}, \mathcal{P})$ we mentioned before in Theorem \ref{thm:d-variateGP}. 
There exists a matrix $M \in \mathbb{R}^{n \times d}$ and two non-negative definite matrices $C = [c]_{jm} \in \mathbb{R}^{n \times n}$ and $\Lambda = [\lambda]_{ab} \in \mathbb{R}^{d \times d}$ such that 
\begin{eqnarray*}
\mathbb{E} \left[ \exp \left( \text{i} \sum_{j = 1}^{n} W_{j, \cdot} Z_{j,\cdot}^{\T} \right) \right] 
&=& \exp \left( - \dfrac{1}{2} \sum_{j,m} W_{j, \cdot} c_{jm} W_{m, \cdot}^{\T} + \text{i} \sum_j W_{j, \cdot} M_{j,\cdot}^{\T} \right)
\\\mathbb{E} \left[ \exp \left( \text{i} \sum_{a = 1}^{d} W_{\cdot, a} Z_{\cdot,a}^{\T} \right) \right] 
&=& \exp \left( - \dfrac{1}{2} \sum_{a,b} W_{\cdot, a} \lambda_{ab} W_{\cdot, b}^{\T} + \text{i} \sum_a W_{\cdot, a} M_{\cdot,a}^{\T} \right)
\end{eqnarray*}
where $W = [w]_{ja} \in \mathbb{R}^{n \times d}$ and $\text{i}$ is the imaginary unit. 
Moreover, we also have the mean value $M = \mathbb{E}[Z]$ and two covariance matrices
\begin{eqnarray*}
c_{jm} &=& \mathbb{E} [(Z_{j,\cdot} - M_{j,\cdot}) (Z_{m,\cdot} - M_{m,\cdot})^{\T}]
\\\lambda_{ab} &=& \mathbb{E} [(Z_{\cdot,a} - M_{\cdot,a})^{\T} (Z_{\cdot,b} - M_{\cdot,b})]. 
\end{eqnarray*}
Assume that the mean matrix $M$ here is a zero matrix, i.e. $\mathbb{E}[Z] = \mathbb{E}[Z \vert t = 0] = 0$, $I_d = \Lambda$ and
\begin{equation*}
C = 
\begin{bmatrix}
    t_1 & t_1 & \cdots & t_1 \\
    t_1 & t_2 & \cdots & t_2 \\
    \vdots  &  \vdots &    & \vdots \\
    t_1 & t_2 & \cdots & t_n
\end{bmatrix}.
\end{equation*}
Hence, $\mathbb{E}[B_t]=0$ for all $t \geq 0$ and 
\begin{align*}
\mathbb{E}[(B_t) (B_t)^{\T}] = d t, \;
\mathbb{E}[(B_t) (B_s)^{\T}] = d \min(s,t),
\\\mathbb{E}[(B_t)^{\T} (B_t)] = t \Lambda, \;
\mathbb{E}[(B_t)^{\T} (B_s)] = \min(s,t) \Lambda.
\end{align*}
Moreover, we have
\begin{align*}
\mathbb{E} [(B_t - B_s) (B_t - B_s)^{\T}]
= \mathbb{E} [B_t B_t^{\T} - 2 B_s B_t^{\T} + B_s B_s^{\T}]
= d \lvert t - s \rvert
\\\mathbb{E} [(B_t - B_s)^{\T} (B_t - B_s)]
= \mathbb{E} [ B_t^{\T} B_t - 2 B_s^{\T} B_t + B_s^{\T} B_s]
= \lvert t - s \rvert \Lambda.
\end{align*}

Note that this $d$-variate Brownian motion $B_t$ still has independent increments since $\mathbb{E} [(B_{t_i} - B_{t_{i-1}}) (B_{t_j} - B_{t_{j-1}})^{\T}] = 0$ and $\mathbb{E} [(B_{t_i} - B_{t_{i-1}})^{\T} (B_{t_j} - B_{t_{j-1}})] = 0$ when $t_i < t_j$ holds for all $0 < t_1 < \cdots < t_n$.

\begin{Remark} Similar to $d$-variate Gaussian white noise, $d$-variate Brownian motion also has independence property along with $T$, but it has correlation along with $d$-variate dimension. Therefore, $d$-variate Brownian motion is also called as variate-dependent Brownian motion or variate-correlated Brownian motion, which is distinct from the "traditional" $d$-dimensional Brownian motion. Actually, the  "traditional" $d$-dimensional Brownian motion is a special case of $d$-variate Brownian motion with diagonal matrix $\Lambda$.
\end{Remark}

As a Brownian motion, we then introduce It\^o lemma for the $d$-variate Brownian motion.
Let $B_t = [B_1(t), \cdots, B_d(t)]$ be the $d$-variate Brownian motion derived in Section \ref{MBM}. 
Then, we have the following lemma.
\begin{Lemma}[It\^o lemma for the $d$-variate Brownian motion]
Let $F$ be a twice continuously differentiable real function on $\mathbb{R}^{d+1}$ and let $\Lambda = [\lambda]_{i,j} \in \mathbb{R}^{d \times d}$ be the covariance matrix for the $d$-variate dimension. Then, 
\begin{align*}
F (t, B_1(t), \cdots, B_d(t)) = F (0, B_1(0), \cdots, B_d(0)) + \sum_{i=1}^d \int_0^t \dfrac{\partial F}{\partial B_i} (s, B_1(s), \cdots, B_d(s)) \dif B_i(s) 
\\
+ \int_0^t \left\{ \dfrac{\partial F}{\partial s} (s, B_1(s), \cdots, B_d(s)) + \dfrac{1}{2} \sum_{i,j=1}^{d} \dfrac{\partial^2 F}{\partial B_i \partial B_j} (s, B_1(s), \cdots, B_d(s)) \lambda_{i,j} \right\} \dif s.
\end{align*}
\end{Lemma}

\begin{proof}
By It\^o lemma and the definition of the $d$-variate Brownian motion, we obtain
\begin{align*}
F (t, B_1(t), \cdots, B_d(t)) &= F (0, B_1(0), \cdots, B_d(0)) + \int_0^t \dfrac{\partial F}{\partial s} (s, B_1(s), \cdots, B_d(s)) \dif s \\ 
& \quad + \sum_{i=1}^d \int_0^t \dfrac{\partial F}{\partial B_i} (s, B_1(s), \cdots, B_d(s)) \dif B_i(s) \\
& \quad + \dfrac{1}{2} \sum_{i,j=1}^{d} \int_0^t \dfrac{\partial^2 F}{\partial B_i \partial B_j} (s, B_1(s), \cdots, B_d(s)) \dif \langle B_i, B_j \rangle (s).
\end{align*}
The proof is complete by $d\langle B_i, B_j \rangle (s) = \lambda_{i,j} \dif s$.
\end{proof}

\section{Application: multivariate Gaussian process regression}\label{section:application}

As a useful application, multi-output prediction using multivariate Gaussian process is a good example. Multivariate Gaussian process provides a solid and unified framework to make the prediction with multiple responses by taking advantage of their correlations. 
As a regression problem,  multivariate Gaussian process regression (MV-GPR) have closed-form expressions for the marginal likelihoods and predictive distributions and thus parameter estimation can adopt the same optimization approaches as used in the conventional Gaussian process \cite{chen2020multivariate}. 

As a summary of MV-GPR in \cite{chen2020multivariate}, the noise-free multi-output regression model is considered and the noise term is incorporated into the kernel function. Given $n$ pairs of observations $\{(x_i,\bm{y}_i)\}_{i=1}^n, x_i \in \mathbb{R}^p, \bm{y}_i \in \mathbb{R}^{d}$, we assume the following model
\begin{equation*}
  \bm{f} \sim  \mathcal{MGP}_{d}(\bm{0},k',\Lambda),  \quad
   \bm{y}_i  =  \bm{f}(x_i), \mbox{ for} \; i = 1,\cdots,n,
\end{equation*}
where $\Lambda$ is an undetermined covariance (correlation) matrix (the relationship between different outputs), $k' = k(x_i,x_j) + \delta_{ij}\sigma_n^2, $
and $\delta_{ij}$ is Kronecker delta.
According to multivariate Gaussian process, it yields that the collection of functions $[\bm{f}(x_1),\ldots,\bm{f}(x_n)]$ follows a matrix-variate Gaussian distribution
$$
[\bm{f}(x_1)^{\T},\ldots,\bm{f}(x_n)^{\T}]^{\T} \sim \mathcal{MN}(\bm{0},K',\Lambda),
$$
where $K'$ is the $n \times n$ covariance matrix of which the $(i,j)$-th element $[K']_{ij} = k'(x_i,x_j)$.
Therefore, the predictive targets 
$$\bm{f}_* = [f_{*1},\ldots,f_{*m}]^\T$$
at the test locations 
$$X_* = [x_{n+1},\ldots,x_{n+m}]^\T$$ 
is given by
 \begin{equation*}
   p(\bm{f}_*|X,Y,X_*) = \mathcal{MN}(\hat{M},\hat{\Sigma},\hat{\Lambda}),
 \end{equation*}
 where 
 $$\hat{M} = K'(X_*,X)^{\T}K'(X,X)^{-1}Y,$$ 
 $$ \hat{\Sigma}  = K'(X_*,X_*)  - K'(X_*,X)^{\T}K'(X,X)^{-1}K'(X_*,X),$$
 and 
 $$\hat{\Lambda} =  \Lambda.$$
Here $K'(X,X)$ is an $n \times n$ matrix of which the $(i,j)$-th element
$[K'(X,X)]_{ij} = k'(x_{i},x_j)$, $K'(X_*,X)$ is an $m \times n$ matrix of which the $(i,j)$-th element
$[K'(X_*,X)]_{ij} = k'(x_{n+i},x_j)$, and $K'(X_*,X_*)$ is an $m \times m$ matrix with
 the $(i,j)$-th element $[K'(X_*,X_*)]_{ij} = k'(x_{n+i},x_{n+j})$. 
 In addition, the expectation and the covariance are obtained,
\begin{eqnarray*}
  \mathbb{E}[\bm{f}_*] &=& \hat{M}=K'(X_*,X)^{\T}K'(X,X)^{-1}Y, \label{pred_mean}\\
  \mathrm{cov}(\mathrm{vec}(\bm{f}^{\T}_*)) &=& \hat{\Sigma}\otimes \hat{\Lambda}  = [K'(X_*,X_*) - K'(X_*,X)^{ \T}K'(X,X)^{-1}K'(X_*,X)] \otimes \Lambda. \label{pred_var}
\end{eqnarray*}

From the view of data science, the hyperparameters involved in the covariance function (kernel) $k'( \cdot , \cdot)$ and the row covariance matrix of MV-GPR need to be estimated from the training data using many approaches \cite{williams1998bayesian}, such as maximum likelihood estimation, maximum a posteriori and Markov chain Monte Carlo \cite{cite:WCKI}.

\section{Conclusion}\label{section:conclusion}
In this paper, we give a proper definition of the multivariate Gaussian process (MV-GP) and some related properties such as strict stationarity and independence of this process. 
We also provide the examples of multivariate Gaussian white noise and multivariate Brownian motion including It\^o lemma and present an useful application of multivariate Gaussian process regression in statistical learning with our definition. 

\section*{Acknowledgements}
The authors would like to thank Dr Youssef El-Khatib for his comments and Dr. Gregory Markowsky for his kind proofreading and very helpful comments.



\bibliographystyle{plainnat}
\bibliography{reference.bib}

\end{document}